\documentclass[12pt,reqno]{amsart}
\usepackage{amsmath, amsthm, amssymb}

\advance \topmargin by -\headheight
\advance \topmargin by -\headsep
     
\evensidemargin \oddsidemargin
\newtheorem{theorem}{Theorem}[section]
\newtheorem{lemma}[theorem]{Lemma}

\theoremstyle{definition}

\theoremstyle{remark}

\numberwithin{equation}{section}

\def\bfa{{\mathbf a}}
\def\bfb{{\mathbf b}}
\def\bfc{{\mathbf c}}
\def\bfd{{\mathbf d}}

\def\bfx{{\mathbf x}}

\def\calA{{\mathcal A}}  

\def\calB{{\mathcal B}}  \def\calBhat{{\widehat \calB}}

\def\calN{{\mathcal N}}

\def\calZ{{\mathcal Z}}

\def\dbN{{\mathbb N}}
\def\dbR{{\mathbb R}}
\def\dbZ{{\mathbb Z}}

\def\gra{{\mathfrak a}}\def\grA{{\mathfrak A}}
\def\grb{{\mathfrak b}}\def\grB{{\mathfrak B}}
\def\grc{{\mathfrak c}}

\def\grJ{{\mathfrak J}}

\def\grm{{\mathfrak m}}\def\grM{{\mathfrak M}}\def\grN{{\mathfrak N}}
\def\grn{{\mathfrak n}}\def\grS{{\mathfrak S}}
\def\grB{{\mathfrak B}}
\def\grK{{\mathfrak K}}

\def\alp{{\alpha}} \def\bfalp{{\boldsymbol \alpha}}
  
\def\gam{{\gamma}} \def\bfgam{{\boldsymbol \gamma}}

\def\del{{\delta}} \def\Del{{\Delta}}

\def\tet{{\theta}} \def\bftet{{\boldsymbol \theta}} \def\Tet{{\Theta}}
\def\kap{{\kappa}}
 \def\Lam{{\Lambda}} 

\def\bfxi{{\boldsymbol \xi}}

\def\sig{{\sigma}}  

\def\Ups{{\Upsilon}} 

\def\ome{{\omega}} \def\Ome{{\Omega}} 
\def\d{{\partial}}
\def\eps{\varepsilon}

\def\le{\leqslant} \def\ge{\geqslant}

\def\d{{\,{\rm d}}}

\def\NN{{\mathbb N}}
\def\ZZ{{\mathbb Z}}

\def\eps{\varepsilon}
\def\al{\alpha}

\def\del{\delta}

\def\rho{\varrho}
\def\phi{\varphi}
\def\gm{\gamma}

\def\Ups{\Upsilon}

\newcommand{\be}[1]{\begin{equation}\label{#1}}
\newcommand{\ee}{\end{equation}}

\newcommand{\rf}[1]{{\rm (\ref{#1})}}

\def\cal{\mathcal}

\begin{document}
\title[Arithmetic harmonic analysis]{Arithmetic harmonic analysis for smooth quartic Weyl sums: three additive equations}
\author[J\"org Br\"udern]{J\"org Br\"udern}
\address{Mathematisches Institut, Bunsenstrasse 3--5, D-37073 G\"ottingen, Germany}
\email{bruedern@uni-math.gwdg.de}
\author[Trevor D. Wooley]{Trevor D. Wooley}
\address{School of Mathematics, University of Bristol, University Walk, Clifton, Bristol BS8 1TW, United Kingdom}
\email{matdw@bristol.ac.uk}
\subjclass[2010]{11D72, 11P55, 11E76}
\keywords{Quartic Diophantine equations, Hardy-Littlewood method.}
\thanks{The authors acknowledge support by Akademie der Wissenschaften zu 
G\"ottingen and Deutsche Forschungsgemeinschaft}
\date{}

\begin{abstract} We establish the non-singular Hasse principle for systems of three 
diagonal quartic equations in $32$ or more variables, subject to a certain rank condition. 
Our methods employ the arithmetic harmonic analysis of smooth quartic Weyl sums and 
also a new estimate for their tenth moment.
\end{abstract}
\maketitle

\section{Introduction} In recent years, investigations concerning the solubility of systems 
of diagonal Diophantine equations via the circle method have been enriched through the 
use of such unconventional elements as thin averages of Fourier coefficients only partially 
of arithmetic nature \cite{BW2007}, and moment estimates of odd order \cite{JEMS}. 
These innovations have been applied in several instances to surmount the barrier imposed  
by the classical scaling principle for suitably entangled systems of diagonal equations. This 
principle suggests that the number of variables required to solve a system should grow in 
proportion to the number of its equations. In particular, the recent work of the authors 
\cite{JEMS} concerning pairs of diagonal quartic equations applies estimates for cubic 
moments of Fourier coefficients to show that $22$ variables suffice to establish the Hasse 
principle. While the corresponding conclusion  for a single quartic equation is available only 
when the number of variables is at least $12$, our work \cite{JEMS} employs, on 
average, only $11$ variables per equation. We now develop such ideas further, and 
provide a flexible approach to the control of large values of Fourier coefficients associated 
with quartic Weyl sums. Once the arithmetic problem at hand is transformed into one in 
which only Fourier coefficients are present, one is at liberty to consider fractional numbers 
of variables, as well as fractional numbers of equations. We illustrate the potential of 
such ideas by investigating the Hasse principle for systems involving three diagonal quartic 
forms.\par

Consider a matrix $(a_{ij})\in\ZZ^{3\times s}$ with the associated system of equations
\be{1.1} \sum_{j=1}^s a_{ij}x_j^4 = 0 \quad (1\le i\le 3). \ee
We verify the Hasse principle for systems of the shape \rf{1.1}, subject to a suitable rank 
condition on $(a_{ij})$, whenever $s\ge 32$. This should be compared with the conclusion 
from \cite[Theorem 1]{BC}, which would furnish the Hasse principle for systems of $r$ 
diagonal quartic equations in $s$ variables only when $s>12r$. While the latter conclusion 
is consistent with the classical scaling principle mentioned above, our new result 
concerning the system \rf{1.1} employs an average of only $10\frac23$ variables per 
equation, and is even more economical than our earlier results \cite{JEMS} for pairs of 
equations. The new methods of this paper also improve the latter work, covering 
essentially all of those cases in $22$ or more variables that had previously defied 
resolution (see \cite{BW2016}).\par

In order to give a precise statement of our result, we introduce some notation. When 
$s\ge 3$ and any collection of three columns of the matrix $(a_{ij})$ is linearly 
independent, we refer to $(a_{ij})$ as being {\em highly non-singular}.  We say that the 
matrix of coefficients $(a_{ij})$ is {\em propitious} when $s\ge 32$ and it has the block 
structure $(A_0,A_1,\ldots,A_7,B)$, in which $A_l\in\ZZ^{3\times 4}$ is highly non-singular 
for each $l$, and $B\in\ZZ^{3\times(s-32)}$. Note that the set of $3\times s$ matrices 
with $s\ge 32$ which fail to be propitious is very thin. Indeed, typical $3\times s$ matrices 
are highly non-singular, and hence also propitious when $s\ge 32$. Finally, given a positive 
number $P$, we denote by $\calN(P)$ the number of integral solutions $\bfx$ of 
(\ref{1.1}) with $|x_j|\le P$ $(1\le j\le s)$.

\begin{theorem}\label{theorem1.1}
Let $s\ge 32$, and suppose that $(a_{ij})\in \ZZ^{3\times s}$ 
is propitious.  Then provided that the system \rf{1.1} has non-singular 
real and $p$-adic solutions for each prime number $p$, one has $\calN(P)\gg P^{s-12}$.
\end{theorem}

We remark that \cite[Theorem 1]{ABC} guarantees the existence of a non-zero $p$-adic solution of the Diophantine system \rf{1.1} provided only that 
$s\ge 25$ and $p\ge 2^{16}$. A familiar $p$-adic compactness argument (see \cite[Theorem 4]{DL1969}) allows one to deduce that
for a propitious system the $p$-adic solubility hypothesis in Theorem \ref{theorem1.1} is void for all $p\ge 2^{16}$,
and one may determine whether or not it possesses non-trivial 
integral solutions with a finite computation.

The novel arithmetic harmonic analysis associated with our proof of Theorem \ref{theorem1.1} depends on
the fourth power moment of certain Fourier coefficients. 
For a continuous function $H:\dbR\rightarrow [0,\infty)$ of period $1$, let 
$$c(n)=\int_0^1H(\alp)e(-n\alp)\d\alp,$$
where as usual we write $e(z)$ for $e^{2\pi iz}$. We relate the
correlation
$$
\int_{[0,1)^3}  H(\al_1)H(\al_2)H(\al_3)H(-\al_1-\al_2-\al_3)\d\bfalp 
$$
to the moment $\sum\limits_{n\in\ZZ}|c(n)|^4$,
and bound the latter by using large values estimates for Fourier coefficients.

Choose a number $\del\in[0,1)$, and let 
$$g_\del(\alp;P,R)=\sum_{\substack{x\in \calA(P,R)\\ \del P< x \le P}} e(\alp x^4),$$
where $\calA(P,R)$ denotes the set of numbers $n\in[1,P]$, all of whose prime divisors 
are at most $R$. When there is no doubt about the choice of parameters, we abbreviate 
$g_\del(\alp;P,R)$ to $g(\al)$. We take $R=P^\eta$ and $H(\al)=|g(\al)|^{8-\nu}$, where 
$\eta$ and $\nu$ are sufficiently small 
positive numbers. An application of H\"older's inequality conveys us from the above correlation to the mean value
$$
\int_{[0,1)^3}  |g(\al_1)g(\al_2)|^8 |g(\al_3)g(\al_1+\al_2+\al_3)|^{8-2\nu}\d\bfalp .
$$
Here the presence of even exponents offers the possibility of replacing the smooth Weyl 
sum $g(\al)$ by its classical cousin 
$$f_\del(\alp;P)=\sum_{\del P< x\le P}e(\alp x^4),$$ 
and one perceives the potential for applying the Hardy-Littlewood method to achieve an essentially optimal estimate.
In this way, in \S 5 we obtain  the estimate contained in the following theorem, which 
provides just adequate space for a subsequent application of the circle method to establish Theorem~\ref{theorem1.1}.

\begin{theorem}\label{theorem1.2}
Suppose that $a_i,b_i$ $(1\le i\le 3)$ are non-zero integers, and that $\del\in(0,1)$. 
Then, whenever $\eta$ and $\nu$ are sufficiently small positive numbers and $1\le R\le P^\eta$, one has
$$\int_{[0,1)^3}|g(a_1\al_1)g(a_2\al_2)g(a_3\al_3)g(b_1\al_1+b_2\al_2+b_3\al_3)|^{8-\nu} 
\d\bfalp \ll P^{20-4\nu}.$$
\end{theorem}

Our proof of Theorem \ref{theorem1.2} involves an analysis of the large values of the Fourier coefficients
$$ \int_0^1|g(\alp)|^{8-\nu} e(-n\al)\d\alp, $$
and this is made to depend on a tenth moment of $g(\al)$. Unfortunately, available estimates for this tenth moment
would fall woefully short of the strength required to press the method home. We therefore reconfigure and enhance
earlier analyses of quartic smooth Weyl sums due to Vaughan \cite{Vau1989} and the present authors \cite{BW2000}.  
In this context, we refer to the number $\Del_t$ as an 
{\it admissible exponent} for the positive even integer $t$ if there exists a positive number $\eta$ 
such that, whenever $1\le R\le P^\eta$, one has
\begin{equation}\label{2.1}
\int_0^1|f_0(\alp;P)^2g_0(\alp;P,R)^{t-2}|\d\alp\ll P^{t-4+\Del_t}.
\end{equation}
Note that, in such circumstances, it follows from orthogonality and a consideration of the 
underlying Diophantine equations that
\be{2.1A}\int_0^1|g_0(\alp;P,R)|^t\d\alp\ll P^{t-4+\Del_t}.\ee

\begin{theorem}\label{theorem1.3}
The number $\Delta_{10}=0.1991466$ is an admissible exponent.
\end{theorem}

We remark that, by applying the methods of Vaughan \cite{Vau1989}, the authors 
\cite{BW2000} obtained the admissible exponent $0.213431$ in place of $0.1991466$, 
this improving on the earlier work of Vaughan \cite{Vau1989, Vau1989b}, which, when appropriately 
combined, delivers the bound \rf{2.1A} in the case $t=10$ with $\Delta_{10}=0.2142036$. For the 
application considered here it is vital
to have at hand admissible exponents $\Delta_8$ and $\Delta_{10}$ with $\Delta_8+ 2\Delta_{10}<1$. In our earlier work
 \cite{BW2000} we showed that $\Delta_8=0.594193$ is admissible. With the numerical value for $\Delta_{10}$
provided by Theorem \ref{theorem1.3}, we obtain  $\Delta_8+ 2\Delta_{10}<0.9925$, leaving barely any space 
to spare in the precision to which we estimate the tenth moment\footnote{Ford \cite{F} 
and Israilov and Allakov \cite{IA} have recorded exponents $\Delta_8$ and $\Delta_{10}$ 
that are smaller than those obtained here. These works are erroneous. See also 
\cite{F2}.}.  

Our basic parameter is $P$, a sufficiently large positive number. In this paper, implicit 
constants in Vinogradov's notation $\ll$ and $\gg$ may depend on $s$ and $\eps$, as well 
as ambient coefficients stemming from  Diophantine systems such as \rf{1.1}. We make 
frequent use of vector notation in the form $\bfx=(x_1,\ldots,x_r)$. Here, the dimension 
$r$ depends on the course of the argument. Occasionally, we abbreviate systems of 
inequalities $0\le a_i\le q$ $(1\le i\le r)$ to $0\le \bfa\le q$, and use $(q,\bfa)$ as a 
shorthand for the largest factor common to the integers $a_1,a_2,\ldots,a_r$ and the 
natural number $q$. Whenever $\eps$ appears in a statement, either implicitly or 
explicitly, we assert that the statement holds for each $\eps>0$. Whenever $R$ appears 
in a statement, it is asserted that there exists a number $\eta>0$ such that this 
statement is true for all $1\le R\le P^\eta$. Whenever $\eps$ occurs in a statement 
involving also $R$, then we allow  $\eta$ to depend on $\eps$. Note that our conventions 
allow us, for example, to conclude that $R^9 \ll P^\eps$. 
 
\section{The tenth moment of smooth quartic Weyl sums} In this section, we shall be 
occupied with the verification of Theorem \ref{theorem1.3}. The new ingredient in our 
treatment is an approach to the exponential sum associated with the difference polynomial
$$ \Psi(z,h,m) = m^{-4} ((z+hm^4)^4 -(z-hm^4)^4) = 8hz(z^2+h^2m^8) $$
that diverges from previous work in several respects. Some notation is required to 
describe the novel features in detail. Whenever $0\le \theta\le 1/4$, we put
$$ M=P^\theta,\quad H= PM^{-4},  \quad Q= PM^{-1}, $$
and introduce the sum
$$ E_0 (\al) = \sum_{1\le h \le H}\sum_{1\le l\le 2P}\sum_{M<m_1<m_2\le MR}
e(8\al lh^3 (m_1^8-m_2^8)). $$
Our first auxiliary lemma supplies an estimate for the mean square of $E_0(\al)$.

\begin{lemma}\label{Esquare} One has
$$\int_0^1 |E_0(\al)|^2\d\al \ll P^{1+\eps}HM^2.$$
\end{lemma}

\begin{proof} The integral on the left hand side of the proposed estimate is equal to the
number of solutions of the Diophantine equation
$$ l_1 h_1^3(m_1^8-m_2^8) = l_2 h_2^3(m_3^8-m_4^8), $$
in which, for $j=1$ and $2$, the variables are subject to the conditions
$$1\le l_j \le 2P, \quad 1\le h_j\le H,\quad M<m_{2j-1}<m_{2j}\le MR.$$
There are  $O(PHM^2R^2)$ choices for $l_2,h_2,m_3,m_4$, and for each such choice, 
the numbers $l_1$, $h_1$ and $m_1^8-m_2^8$ are divisors of the non-zero integer 
$l_2h_2^3(m_3^8-m_4^8)$. A familiar estimate for the number of divisors now shows 
that the number of choices for $l_1,h_1,m_1$ and $m_2$ is bounded by $O(P^\eps)$, 
and the lemma follows.
\end{proof}

The next lemma is the key to our new tenth moment estimate.

\begin{lemma}\label{EG} One has
$$\int_0^1 |E_0(\al)^2f_0(\al;2Q)^4| \d\al \ll Q^{5+\eps}.$$
\end{lemma}

\begin{proof} By Weyl's differencing technique \cite[Lemma 2.3]{Vau1997}, one finds that
$$|f_0(\al;2Q)|^4 \ll Q^3 + Q \sum_{0<|n|\le 32Q^4} c(n) e(\al n), $$
where the coefficients $c(n)$ are certain integers satisfying $c(n)\ll |n|^\eps$. Write
$$\rho(n)=\int_0^1|E_0(\al)|^2e(\al n)\d\al .$$
Then it follows that
$$\int_0^1|E_0(\al)^2f_0(\al;2Q)^4|\d\al \ll Q^3\rho(0)+
Q\sum_{0<|n|\le 32Q^4}c(n)\rho(n).$$
By orthogonality, one has $\rho(n)\ge 0$. Furthermore, Lemma \ref{Esquare} supplies the 
bound $\rho(0)\ll P^{1+\eps}HM^2$. Thus, we deduce that
\begin{align*}
\int_0^1|E_0(\al)^2f_0(\al;2Q)^4|\d\al& \ll Q^3P^{1+\eps}HM^2+QP^\eps 
\sum_{n\in\ZZ} \rho(n)\\
&\ll P^\eps(Q^3PHM^2+QE_0(0)^2)\ll Q^{5+2\eps}.\end{align*}
\end{proof}

\begin{lemma}\label{lemma2.1} The exponents $\Delta_8= 0.594193$ and
 $\Del_{12}=0$ are admissible.
\end{lemma}

\begin{proof} The desired conclusion concerning $\Delta_8$ follows from 
\cite[Theorem 2]{BW2000} and the discussion surrounding the table of exponents on 
\cite[page 393]{BW2000}. Meanwhile, the upper bound (\ref{2.1}) when $t=12$ is a 
consequence of \cite[Lemma 5.2]{Vau1989}.
\end{proof}

We initiate our estimation of the tenth moment by choosing an admissible value for 
$\Delta_{10}$. That such values exist follows from the trivial bounds for $f_0$ and  
$g_0$. For the rest of this section, we work with the sums $g_0(\al;P,R)$ and $f_0(\al;P)$ 
only, and abbreviate these to $g(\al)$ and $f(\al)$, respectively. We put 
$g_\flat(\al)=g_0(\al;2Q,R)$, and for the sake of concision, for positive even integers $t$, we write
$$U_t=\int_0^1|g_\flat(\alp)|^t\d\alp.$$
Further, we require the exponential sum 
$$F_1(\al)=\sum_{1\le h\le H}\sum_{M<m\le MR}\sum_{1\le z\le 2P}
e(8\al hz(z^2+h^2m^8)).$$

\begin{lemma}\label{lemma2.4} Suppose that $\Del_8$ and $\Del_{10}$ are admissible 
exponents satisfying
$$\tfrac{1}{2}<\Del_8<\tfrac{3}{5},\quad \tfrac{1}{10}<\Del_{10}<\tfrac{1}{4}
\quad \text{and}\quad \tfrac{3}{2}\Del_8-\tfrac{5}{7}<\Del_{10}<2\Del_8-
\tfrac{27}{28}.$$
Put
$$\tet=\max\left\{ \frac{3}{17},\frac{7+2\Del_{10}-3\Del_8}{33+2\Del_{10}-3\Del_8}\right\},$$
and define $\Del'_{10}=\Del_8(1-\tet)+4\tet-1$. Then whenever $\Del>\Del'_{10}$, the 
exponent $\Del$ is admissible.
\end{lemma}

\begin{proof} Our starting point is an application of a suitable version of the {\it fundamental 
lemma}. Thus, as a consequence of \cite[Lemma 2.3]{Woo1992} in combination with the argument of the 
proof of \cite[Lemma 3.1]{Woo1992} (see \cite[Lemma 2.1]{VWfurth1}),
\be{2.5} \int_0^1 |f(\al)^2g(\al)^8|\d\al \ll P^\eps M^7 (PMQ^{4+\Del_8} + T) \ee
where
\be{2.6}
T = \int_0^1 F_1(\al) |g_\flat(\al)|^8\d\al. \ee

\par By Cauchy's inequality, 
$$|F_1(\al)|^2 \le HMR \sum_{1\le h \le H}\sum_{M<m\le MR}
\bigg| \sum_{1\le z\le 2P}e(\al\Psi(z,h,m))\bigg|^2.$$
Here, we open the square  and rewrite it as a double sum over $z_1$ and $z_2$, say. The 
substitutions $z=z_1+z_2$ and $l=z_1-z_2$ then yield 
$$\bigg| \sum_{1\le z\le 2P}e(\al\Psi(z,h,m))\bigg|^2=\sum_{|l|\le 2P} 
\sum_{z\in\calB(l)}e(\al hl(6z^2+2l^2 + 8h^2m^8)),$$
in which $\calB(l)$ denotes the set of all integers $z$ with  $1\le z\pm l\le 4P$ and  
$z\equiv l\bmod 2$. Separation of the term $l=0$ delivers the inequality
$$|F_1(\al)|^2\ll P^{1+\eps}H^2M^2+P^\eps HM\sum_{\substack{1\le h\le H\\ 1\le l\le 2P}}
\bigg|\sum_{M<m\le MR} \sum_{z\in\calB(l)}
e(\al hl(6z^2+8h^2m^8))\bigg|.$$
Yet another application of Cauchy's inequality now produces the bound
$$|F_1(\al) |^2\ll P^{1+\eps}H^2M^2+P^\eps HM\big(D(\al)E(\al)\big)^{1/2},$$
in which
$$D(\al) = \sum_{1\le h\le H}\sum_{1\le l\le 2P} \bigg| \sum_{z\in\calB(l)}
e(6\al hlz^2)\bigg|^2$$
and
$$ E(\al) = \sum_{1\le h\le H} \sum_{1\le l\le 2P} \bigg| \sum_{M<m\le MR}
e(8\al lh^3m^8)\bigg|^2. $$
On substituting the last inequality for $|F_1(\al)|^2$ into \rf{2.6}, we infer that
\be{2.7} T\ll P^{1/2+\eps}HMQ^{4+\Del_8} + P^\eps(HM)^{1/2} T_1,\ee
where
\be{2.8} T_1 = \int_0^1 \big( D(\al)E(\al)\big)^{1/4} |g_\flat(\al)|^8\d\al .\ee

We apply the Hardy-Littlewood method to estimate $T_1$. For integers $a,q$ with 
$0\le a\le q\le P$ and $(a,q)=1$, let  $\grN(q,a)$ denote the set of all $\al\in[0,1)$
with $|q\al-a|\le PQ^{-4}$, and let $\grN$ denote the union of these intervals. Note 
that this union is disjoint. Define the function $\Omega: [0,1) \to [0,1]$  by  
$$ \Omega(\al) = (q + Q^4 |q\al-a|)^{-1}  \quad (\al\in \grN(q,a)), $$
and put $\Omega(\al)=0$ when $\al\not\in\grN$. 

By Dirichlet's theorem on Diophantine approximation, whenever $\al\in[0,1)$, there are 
 integers $a,q$ with $0\le a\le q\le Q^4P^{-1}$,  $(a,q)=1$ and $|q\al-a|\le PQ^{-4}$. 
 Moreover, although our sum $D(\al)$ differs in detail from that used by Vaughan 
\cite{Vau1989} in his equation (3.2), the proof of \cite[Lemma 3.1]{Vau1989} applies to 
our sum as well and yields the same estimate. We therefore conclude that the bound
$$ D(\al) \ll P^{2+\eps}H + P^{3+\eps}H \Omega(\al) $$ 
holds for all $\al\in[0,1)$. Consequently, we deduce from \rf{2.8} that
\be{2.9} T_1 \ll
 P^{1/2+\eps}H^{1/4} T_2 + P^{3/4+\eps}H^{1/4}T_3,\ee
where
$$ T_2 = \int_0^1 E(\al)^{1/4} |g_\flat(\al)|^8\d\al \quad \text{and}\quad T_3= \int_\grN 
\big(\Omega(\al) E(\al)\big)^{1/4} |g_\flat(\al)|^8\d\al. $$

The estimation of $T_2$ will involve the application of Lemma \ref{EG}. An inspection of the 
definitions of $E_0(\al)$ and $E(\al)$ reveals that
\begin{equation}\label{2.fix}
E(\al)\ll |E_0(\al)|+PHMR.
\end{equation}
As a first bound for $T_2$, we then have
$$T_2\ll P^\eps (PHM)^{1/4}U_8+\int_0^1 |E_0(\al)|^{1/4}|g_\flat(\al)|^8\d\al.$$
Let
$$V=\int_0^1 |E_0(\al)^{2}g_\flat(\al)^4|\d\al.$$
Then, a further application of H\"older's inequality yields the bound
$$ T_2 \ll P^\eps(PHM)^{1/4}U_8+U_8^{5/8}U_{10}^{1/4}V^{1/8}. $$
We infer from Lemma \ref{EG} via orthogonality that
\begin{equation}\label{vest}
V\ll (PHM^2)^2Q^{1+\eps},
\end{equation} 
and so by applying \rf{2.1A}, we deduce that
$$P^{1/2}H^{1/4}T_2\ll P^{3/4+\eps}(HM)^{1/2}Q^4
\Big( M^{-1/4}Q^{\Delta_8}+Q^{(5\Delta_8+2\Delta_{10}+1)/8}\Big).$$
However, the hypothesis $\Del_{10}>\tfrac{3}{2}\Del_8-\tfrac{5}{7}$ ensures that
$$\frac{3\Del_8-2\Del_{10}-1}{3\Del_8-2\Del_{10}+1}<
\frac{7+2\Del_{10}-3\Del_8}{33+2\Del_{10}-3\Del_8}.$$
Hence we have
$$\tet>\frac{3\Del_8-2\Del_{10}-1}{3\Del_8-2\Del_{10}+1},$$
so that
$$M^2\ge Q^{3\Del_8-2\Del_{10}-1}.$$
We thus conclude that
\be{2.12}
P^{1/2}H^{1/4}T_2\ll P^{3/4+\eps}(HM)^{1/2}Q^{4+(5\Delta_8+2\Delta_{10}+1)/8}.
\ee

As our first step in estimating $T_3$, we apply (\ref{2.fix}) to deduce that
$$T_3\ll T_4+(PHMR)^{1/4}T_5,$$
where
$$T_4=\int_\grN \left( \Ome(\alp)|E_0(\alp)|\right)^{1/4}|g_\flat(\alp)|^8\d\alp \quad \text{and}\quad 
T_5=\int_\grN \Ome(\alp)^{1/4}|g_\flat(\alp)|^8\d\alp .$$
Write
$$W=\int_\grN \Ome(\alp)|g_\flat(\alp)|^4\d\alp .$$
Then an application of \cite[Lemma 2]{B87} confirms the estimate
$$W\ll Q^{\eps-4}(PQ^2+Q^4)\ll Q^\eps.$$
H\"older's inequality therefore combines with (\ref{2.1A}), (\ref{vest}) and Lemma \ref{lemma2.1} 
to give
\begin{align*}
T_4&\le V^{1/8}W^{1/4}U_{10}^{1/2}U_{12}^{1/8}\\
&\ll P^\eps \left( (PHM^2)^2Q\right)^{1/8}(Q^{6+\Del_{10}})^{1/2}(Q^8)^{1/8},
\end{align*}
whence
\begin{equation}\label{2.T4}
P^{3/4+\eps}H^{1/4}T_4\ll P^{1+\eps}(HM)^{1/2}Q^{4+(4\Del_{10}+1)/8}.
\end{equation}
In like manner, another application of H\"older's inequality yields the bound
$$T_5\ll W^{1/4}U_8^{1/4}U_{10}^{1/2}\ll P^\eps (Q^{4+\Del_8})^{1/4}
(Q^{6+\Del_{10}})^{1/2},$$
whence
\begin{equation}\label{2.T5}
P^{1+\eps}H^{1/2}M^{1/4}T_5\ll P^{1+\eps}(HM)^{1/2}Q^4
\left( M^{-1/4}Q^{(\Del_8+2\Del_{10})/4}\right).
\end{equation}
By combining (\ref{2.T4}) and (\ref{2.T5}), we conclude that
$$P^{3/4}H^{1/4}T_3\ll P^{1+\eps}(HM)^{1/2}Q^4\left( Q^{(4\Del_{10}+1)/8}+M^{-1/4}
Q^{(\Del_8+2\Del_{10})/4}\right).$$
The hypotheses of the statement of the lemma imply that
$$\tet\ge \frac{3}{17}>\frac{1}{11}\ge \frac{2\Del_8-1}{2\Del_8+1},$$
so that $M\ge Q^{\Del_8-1/2}$. Thus we conclude that
\begin{equation}\label{2.T6}
P^{3/4}H^{1/4}T_3\ll P^{1+\eps}(HM)^{1/2}Q^{4+(4\Del_{10}+1)/8}.
\end{equation}

We may now collect together our various estimates, first combining \rf{2.9}, \rf{2.12} 
and (\ref{2.T6}), and substituting the result into (\ref{2.7}) to obtain the bound
$$T\ll P^{1+\eps}MHQ^{4+\Del_8} \big(P^{-1/2}+P^{-1/4}Q^{(1+2\Del_{10}-3\Del_8)/8}
+Q^{(1+4\Del_{10}-8\Del_8)/8}\big).$$
Since $\tet\ge \frac{3}{17}>\frac{1}{8}$, one has $HP^{-1/2}\le 1$, and the bound 
$$HP^{-1/4}Q^{(1+2\Del_{10}-3\Del_8)/8}\le 1$$
follows in its turn from the hypothesis that 
$$\tet\ge \frac{7+2\Del_{10}-3\Del_8}{33+2\Del_{10}-3\Del_8}.$$
Meanwhile, since we suppose that $\Del_{10}<2\Del_8-\tfrac{27}{28}$, one finds from 
the hypothesis $\tet\ge \frac{3}{17}$ that 
$$HQ^{(1+4\Del_{10}-8\Del_8)/8}<HQ^{-5/14}\le 1.$$
We therefore deduce that $T\ll P^{1+\eps}MQ^{4+\Del_8}$, and on substituting into \rf{2.5}, we 
obtain the bound
$$\int_0^1|f(\al)^2g(\al)^8|\d\al \ll P^{1+\eps} M^8Q^{4+\Delta_8}= 
P^{6+\Del'_{10}+\eps},$$
where $\Delta'_{10}=\Del_8(1-\theta)+4\theta -1$. It follows that whenever $\Del>\Del'_{10}$, then 
$\Del$ is admissible, and so the proof of the lemma is complete.
\end{proof}

We are now equipped to describe the iteration that yields the admissible exponent recorded in 
Theorem \ref{theorem1.3}. We recall from Lemma \ref{lemma2.1} that the exponent 
$\Del_8=0.594193$ is admissible. Also, from the work of Vaughan \cite{Vau1989} and the authors 
\cite{BW2000}, there exists an admissible exponent $\Del_{10}$ smaller than $0.22$. 
Suppose then 
that an admissible exponent $\Del_{10}$ has been established satisfying
$$0.2241\ldots =2\Del_8-\tfrac{27}{28}>\Del_{10}>
\tfrac{3}{2}\Del_8-\tfrac{5}{7}=0.1770\ldots .$$
It follows that Lemma \ref{lemma2.4} then applies with
$$\tet=\frac{7+2\Del_{10}-3\Del_8}{33+2\Del_{10}-3\Del_8},$$
and that any exponent $\Del'_{10}$ exceeding $\Del_8(1-\tet)+4\tet-1$ is also admissible. 
On iterating this treatment, one finds a decreasing sequence of admissible exponents converging 
to the larger root $\Delta_{10}^*$ of the equation 
$$ \Delta^*_{10}  = \Del_8 - 1 + (4-\Del_8) \frac{7  + 2\Delta^*_{10}-3 \Delta_8}{ 33+ 
 2 \Delta^*_{10}-3 \Delta_8}. $$
On using the value for $\Del_8$ recorded in Lemma \ref{lemma2.1}, one readily confirms that 
$\Del_{10}^*$ satisfies the equation
$$2(\Del_{10}^*)^2+(27-3\Del_8)\Del_{10}^*+5-17\Del_8=0,$$
whence
$$ \Del^*_{10}=\tfrac14 \Big(3\Del_8-27+\sqrt{689-26\Del_8+9 \Del_8^2}\Big) = 
0.199146547\ldots .$$
Given any positive number $\del$, this iteration yields an admissible exponent $\Del_{10}^\dagger$, 
satisfying $\Del_{10}^*<\Del_{10}^\dagger<\Del_{10}^*+\del$, after a number of iterations 
bounded solely in terms of $\del$. Consequently, keeping in mind our conventions concerning 
$\eps$ and $R$, it follows that
$$\int_0^1|f(\alp)^2g(\alp)^8|\d\alp \ll P^{6+\Del^*_{10}+\eps}.$$
We deduce that the exponent $\Del_{10}$ is admissible whenever $\Del_{10}>\Del_{10}^*$, 
and thus we arrive at the conclusion of Theorem \ref{theorem1.3}. 

\section{Large values estimates} Our next task is to provide a proof of the mixed fractional moment 
estimate recorded in Theorem \ref{theorem1.2}. Within this and the next two sections, we fix a choice 
of $\del\in(0,1)$ once and for all, and then adumbrate $f_\del(\al;P)$ to $f(\al)$ and $g_\delta(\al;P,R)$
to $g(\al)$. Finally, according to Theorem \ref{theorem1.3} and Lemma \ref{lemma2.1}, we are at liberty 
to suppose that $\Del_8$ and $\Del_{10}$ are admissible exponents satisfying the inequalities
$$ \Del_8 \le 0.594193 \quad \text{and} \quad \Del_{10}\le 0.1991466.$$ 

When $0\le \tau\le 1$, we define the Fourier coefficient
\begin{equation}\label{3.1}
\psi_\tau(n)=\int_0^1|g(\alp)|^{8-\tau}e(-n\alp)\d\alp .
\end{equation}
Also, when $T>0$, we write
\begin{equation}\label{Mtau}
M_\tau(T)=\sum_{\substack{|n|\le P^5\\ T<|\psi_\tau(n)|\le 2T}}|\psi_\tau(n)|^4.
\end{equation}
By applying the triangle inequality to (\ref{3.1}) in combination with H\"older's inequality, one 
obtains the bound
$$\psi_\tau(n)\le \psi_\tau(0)\le \biggl( \int_0^1|g(\alp)|^8\d\alp\biggr)^{1-\tau/8}
\ll P^{4+\Del_8},$$
and thus we may restrict attention to values of $T$ with $T\le P^5$.\par

We now seek to bound $M_\tau(T)$ when $1\le T\le P^5$. Define $\calZ_T$ to be 
the set of integers $n$ with $|n|\le P^5$ such that $T<|\psi_\tau(n)|\le 2T$, and write 
$Z_T=\text{card}(\calZ_T)$. For each $n\in \calZ_T$, we take $\ome_n=1$ when 
$\psi_\tau (n)>0$, 
and $\ome_n=-1$ when $\psi_\tau(n)<0$, and then define
$$K_T(\alp)=\sum_{n\in \calZ_T}\ome_ne(-n\alp).$$
Thus we have
\begin{align}
\int_0^1|g(\alp)|^{8-\tau}K_T(\alp)\d\alp &=\sum_{n\in \calZ_T}\ome_n\int_0^1
|g(\alp)|^{8-\tau}e(-n\alp)\d\alp \notag \\
&=\sum_{n\in \calZ_T}|\psi_\tau(n)|>TZ_T.\label{3.2}
\end{align}

\par Before announcing our basic large values estimates, we recall that as an immediate 
consequence of \cite[Lemma 2.1]{KW2010}, one has
\begin{equation}\label{3.3}
\int_0^1|g(\alp)^4K_T(\alp)^2|\d\alp \le\int_0^1|f(\alp)^4K_T(\alp)^2|\d\alp  \ll 
P^3Z_T+P^{2+\eps}Z_T^{3/2}.
\end{equation}
Finally, we introduce the exponents
\begin{equation}\label{kap1}
\kap_1(\tau)=11-(2-\Del_{10})\tau
\end{equation}
and
\begin{equation}\label{kap2}
\kap_2(\tau)=19+\Del_8+2\Del_{10}-(4-2\Del_8+2\Del_{10})\tau .
\end{equation}

\begin{lemma}\label{lemma3.2} Let $0<\tau\le 1$ and $1\le T\le P^5$.  Then one has
$$Z_T\ll P^\eps (P^{\kap_r(\tau)}T^{-2r}+P^{2\kap_r(\tau)-2}T^{-4r})\quad (r=1,2).$$
\end{lemma}

\begin{proof} In the looming discussion we drop mention of $T$ and $\tau$ from our 
various notations. When $r\in\{1,2\}$ and $s\in\NN$ is even, define
$$I_r=\int_0^1|g(\alp)^4K(\alp)^{2r}|\d\alp\quad \text{and}\quad 
J_s=\int_0^1|g(\alp)|^s\d\alp .$$
As an immediate consequence of \rf{2.1A}, Lemma \ref{lemma2.1} and Theorem 
\ref{theorem1.3}, one has
\begin{equation}\label{3.4}
J_8\ll P^{4+\Del_8},\quad J_{10}\ll P^{6+\Del_{10}}\quad \text{and}\quad 
J_{12}\ll P^8.
\end{equation}
Then an application of H\"older's inequality shows in the first instance that
$$\int_0^1|g(\alp)|^{8-\tau}K(\alp)\d\alp \le I_1^{1/2}J_{10}^{\tau/2}
J_{12}^{(1-\tau)/2},$$
and by means of \rf{3.2}, \rf{3.3} and \rf{3.4}, we infer the bound
\begin{align*}
TZ<\int_0^1|g(\alp)|^{8-\tau}K(\alp)\d\alp &\ll P^\eps(P^3Z+P^2Z^{3/2})^{1/2}
(P^{6+\Del_{10}})^{\tau/2}(P^8)^{(1-\tau)/2}\\
&\ll P^\eps \left( (P^{\kap_1(\tau)}Z)^{1/2}+(P^{2\kap_1(\tau)-2}Z^3)^{1/4}\right) .
\end{align*}
The claimed estimate with $r=1$ follows on disentangling this bound.\par

Meanwhile, another application of H\"older's inequality yields
$$\int_0^1|g(\alp)|^{8-\tau}K(\alp)\d\alp \le I_2^{1/4}J_8^{(1+2\tau)/4}
J_{10}^{(1-\tau)/2}.$$
Hence, by applying a trivial estimate for $K(\alp)$ in combination with (\ref{3.2}), 
(\ref{3.3}) and (\ref{3.4}), we infer that
\begin{align*}
TZ&\ll P^\eps (P^3Z^3+P^2Z^{7/2})^{1/4}(P^{4+\Del_8})^{(1+2\tau)/4}
(P^{6+\Del_{10}})^{(1-\tau)/2}\\
&\ll P^\eps \left( (P^{\kap_2(\tau)}Z^3)^{1/4}+(P^{2\kap_2(\tau)-2}Z^7)^{1/8}\right).
\end{align*}
The claimed estimate with $r=2$ follows on disentangling this bound.
\end{proof}

We require large values estimates of similar type for related mean values associated with 
a restriction to a set of minor arcs. Define the major arcs $\grM$ to be the union of the 
intervals
\begin{equation}\label{Major}
\grM(q,a)=\{ \alp\in [0,1):|q\alp-a|\le P^{-7/2}\},
\end{equation}
with $0\le a\le q\le P^{1/2}$ and $(a,q)=1$, and then put $\grm=[0,1)\setminus \grM$. 
When $\grB\subseteq [0,1)$ is measurable, we define
$$\Psi_\grB(n)=\int_\grB |f(\alp)^2g(\alp)^6|e(-n\alp)\d\alp ,$$
and when $T>0$, define
\begin{equation}\label{Mzero}
M_0(T)=\sum_{\substack{|n|\le P^5\\ T<|\Psi_\grm(n)|\le 2T}}|\Psi_\grm(n)|^4.
\end{equation}
Define $\calZ_0$ to be the set of integers $n$ with $|n|\le P^5$ for which 
$T<|\Psi_\grm(n)|\le 2T$, and write $Z_0=Z_0(T)$ for $\text{card}(\calZ_0)$. For each 
$n\in \calZ_0$, we take $\ome_n=1$ when $\Psi_\grm(n)>0$, and we put $\ome_n=-1$ when 
$\Psi_\grm(n)<0$. Also, 
we define
$$K_0(\alp)=\sum_{n\in \calZ_0}\ome_ne(-n\alp).$$
Then, as in \rf{3.2}, one obtains
\begin{equation}
\int_\grm |f(\alp)^2g(\alp)^6|K_0(\alp)\d\alp =\sum_{n\in \calZ_0}|\Psi_\grm(n)|>TZ_0.\label{3.5} 
\end{equation}
Before announcing our large values estimates for $\Psi_\grm(n)$, we recall the definitions 
(\ref{kap1}) and (\ref{kap2}) of $\kap_1(\tau)$ and $\kap_2(\tau)$. 

\begin{lemma}\label{lemma3.3} Let $1\le T\le P^5$. Then one has
$$Z_0\ll P^\eps (P^{\kap_1(0)+\Del_{10}-1/4}T^{-2}+P^{2\kap_1(0)+2\Del_{10}-5/2}
T^{-4})$$
and
$$Z_0\ll P^\eps (P^{\kap_2(0)}T^{-4}+P^{2\kap_2(0)-2}T^{-8}).$$
\end{lemma}

\begin{proof}
Define
$${\cal J}=\int_\grm |f(\alp)^2g(\alp)^{10}|\d\alp .$$
An enhanced version of Weyl's inequality (see \cite[Lemma 3]{Vau1986b}) shows that
$$\sup_{\alp\in \grm}|f(\alp)|\ll P^{7/8+\eps},$$
and so we deduce via (\ref{3.4}) that
$${\cal J}\le \left( \sup_{\alp\in \grm}|f(\alp)|\right)^2J_{10}\ll 
P^{8+\Del_{10}-1/4+\eps}. $$
An application of Schwarz's inequality shows that
$$\int_\grm|f(\alp)^2g(\alp)^6|K_0(\alp)\d\alp \le \biggl( 
\int_0^1|f(\alp)g(\alp)K_0(\alp)|^2\d\alp \biggr)^{1/2}{\cal J}^{1/2}.$$
In view of (\ref{3.3}) and  (\ref{3.5}), another application of Schwarz's inequality 
yields 
$$TZ_0<\int_\grm |f(\alp)^2g(\alp)^6|K_0(\alp)\d\alp \ll P^\eps 
(P^3Z_0+P^2Z_0^{3/2})^{1/2}(P^{8+\Del_{10}-1/4})^{1/2}.$$
The first of the claimed estimates follows by disentangling this bound. For the second 
we proceed just as in the proof of Lemma \ref{lemma3.2} in the case $r=2$, noting that the mean value 
estimates for $J_8$ and $J_{10}$ should in this instance be replaced by the estimates
$$\int_0^1|f(\alp)^2g(\alp)^6|\d\alp \ll P^{4+\Del_8}\quad \text{and}\quad 
\int_0^1|f(\alp)^2g(\alp)^8|\d\alp \ll P^{6+\Del_{10}},$$
available via (\ref{2.1}). This completes the proof of the lemma.
\end{proof}

\section{Fourier coefficients and their moments} Our goal in this section is the proof of 
an estimate for a certain mixed moment of Fourier coefficients associated with quartic 
Weyl sums. This we achieve by employing our large values estimates of the previous 
section so as to bound the quantities $M_\tau(T)$ and $M_0(T)$ defined in 
(\ref{Mtau}) and (\ref{Mzero}). We proceed in stages. In what follows, we make use of 
a positive number $\tau$ satisfying
\be{4.0} 40\tau\le \min\{ 1-4\Del_{10},1-2\Del_{10}-\Del_8\}.\ee

\begin{lemma}\label{lemma4.1}
Suppose that $\tau$ is a positive number satisfying {\rm (\ref{4.0})}. Then
$$\sum_{|n|\le P^5}|\psi_\tau(n)|^4\ll P^{20-\tau}\quad \text{and}\quad 
\sum_{|n|\le P^5}|\Psi_\grm(n)|^4\ll P^{20-9\tau}.$$
\end{lemma}

\begin{proof} Observe first that
$$\sum_{\substack{|n|\le P^5\\ |\psi_\tau(n)|>1}}|\psi_\tau(n)|^4\le 
\sum_{\substack{l=0\\ 2^l\le P^5}}^\infty M_\tau(2^l).$$
Thus, for some number $T$ with $1\le T\le P^5$, one has
$$\sum_{|n|\le P^5}|\psi_\tau(n)|^4\ll P^5+(\log P)M_\tau(T)\ll P^5+P^\eps T^4Z_T.$$
Should $T$ satisfy the bound $T\le P^{9/2}$, then it follows from the estimate supplied by 
Lemma \ref{lemma3.2} with $r=1$ that 
\begin{align*}
\sum_{|n|\le P^5}|\psi_\tau(n)|^4&\ll P^5+P^\eps \left( 
P^{\kap_1(\tau)}T^2+P^{2\kap_1(\tau)-2}\right)\\
&\ll P^{20-(2-\Del_{10})\tau+\eps}\ll P^{20-\tau}.
\end{align*}
Meanwhile, when $P^{9/2}<T\le P^5$, we discern from Lemma \ref{lemma3.2} with 
$r=2$ that
\begin{align*}
\sum_{|n|\le P^5}|\psi_\tau(n)|^4&\ll P^5+P^\eps \left( P^{\kap_2(\tau)}+
P^{2\kap_2(\tau)-2}T^{-4}\right) \\
&\ll P^\eps \left(P^{\kap_2(\tau)}+P^{2\kap_2(\tau)-20}\right).
\end{align*}
In view of our hypotheses concerning $\tau$, one finds that
$$\kap_2(\tau)\le 19+\Del_8+2\Del_{10}\le 20-40\tau$$
and
$$2\kap_2(\tau)-20\le 18+2\Del_8+4\Del_{10}\le 20-80\tau.$$
The first conclusion of the lemma is now immediate.\par

In like manner, one finds that for some number $T$ with $1\le T\le P^5$, one has
$$\sum_{|n|\le P^5}|\Psi_\grm(n)|^4\ll P^5+(\log P)M_0(T)\ll P^5+P^\eps T^4Z_0(T).$$
Should one have $T\le P^{9/2}$, then it follows from the first 
estimate of Lemma \ref{lemma3.3} that one has
\begin{align*}
\sum_{|n|\le P^5}|\Psi_\grm(n)|^4&\ll P^5+P^\eps \left( P^{11+\Del_{10}-1/4}T^2+
P^{20+2\Del_{10}-1/2}\right) \\
&\ll P^{\eps}\left( P^{20+\Del_{10}-1/4}+P^{20+2\Del_{10}-1/2}\right).
\end{align*}
Meanwhile, when $P^{9/2}<T\le P^5$, the second estimate of Lemma \ref{lemma3.3} 
yields
\begin{align*}
\sum_{|n|\le P^5}|\Psi_\grm(n)|^4&\ll P^5+P^\eps\left( P^{\kap_2(0)}+
P^{2\kap_2(0)-2}T^{-4}\right)\\
&\ll P^\eps \left( P^{19+\Del_8+2\Del_{10}}+P^{18+2\Del_8+4\Del_{10}}\right).
\end{align*}
Thus, in all cases, our hypotheses concerning $\tau$ ensure that
$$\sum_{|n|\le P^5}|\Psi_\grm(n)|^4\ll P^{20-10\tau+\eps},$$
and the second conclusion of the lemma follows.
\end{proof}

\begin{lemma}\label{lemma4.2}
Let $a$ and $b$ be non-zero integers, and suppose that the positive number $\tau$ satisfies \rf{4.0}. 
Then, one has
$$\sum_{|n|\le P^5}\psi_0(an)^2\psi_\tau(bn)^2\ll P^{20-2\tau}.$$
\end{lemma}

\begin{proof} Observe that, by orthogonality and a consideration of the underlying 
Diophantine equations, one has $\psi_0(an)=0$ whenever $|n|>P^{9/2}$. In addition,
$$\psi_0(n)=\int_0^1|g(\alp)|^8e(-n\alp)\d\alp \le \int_0^1 |f(\alp)^2g(\alp)^6|
e(-n\alp)\d\alp =\Psi_{[0,1)}(n),$$
whence $\psi_0(n)\le \Psi_\grm(n)+\Psi_\grM(n)$. Thus,
$$\psi_0(n)^2\le 2(\Psi_\grm(n)^2+\Psi_\grM(n)^2),$$
and we deduce that 
\begin{equation}\label{4.A}
\sum_{|n|\le P^5}\psi_0(an)^2\psi_\tau(bn)^2\ll \Xi(\grM)+\Xi(\grm),
\end{equation}
where
$$\Xi(\grB)=\sum_{|n|\le P^{9/2}}\Psi_\grB(an)^2\psi_\tau(bn)^2.$$

\par On the one hand, by Cauchy's inequality and Lemma \ref{lemma4.1}, we have
\begin{align*}
\Xi(\grm)&\le \biggl( \sum_{|n|\le P^5}\Psi_\grm(n)^4\biggl)^{1/2}
\biggl( \sum_{|n|\le P^5}\psi_\tau(n)^4\biggr)^{1/2}\\
&\ll (P^{20-9\tau})^{1/2}(P^{20-\tau})^{1/2}\ll P^{20-5\tau}.
\end{align*}
On the other hand, the adjuvant Lemma \ref{lemmahash} provided in the appendix 
combines with the triangle inequality to give $\Psi_\grM(an)=O(P^4)$. Thus, as a 
consequence of Bessel's inequality, one has
$$\Xi(\grM)\ll (P^4)^2\sum_{|n|\le P^5}\psi_\tau(bn)^2\ll P^8
\int_0^1|g(\alp)|^{16-2\tau}\d\alp .$$
By \rf{2.1A} and Lemma \ref{lemma2.1}, we now infer that
$$\Xi(\grM)\ll P^{12-2\tau} \int_0^1|g(\alp)|^{12}\d\alp \ll P^{20-2\tau},$$
and thus it follows that
$\Xi(\grM)+\Xi(\grm)\ll P^{20-2\tau}$.
The conclusion of the lemma is now immediate from (\ref{4.A}).
\end{proof}

\section{The transition to moments of smooth Weyl sums}
In this section we establish Theorem \ref{theorem1.2}. With this end in view, we put
$$ \psi_\tau(m;l)=\begin{cases} \psi_\tau(m/l),&\text{when $l|m$,}\\
0,&\text{otherwise.}\end{cases} $$
Suppose that $a_i,b_i$ $(1\le i\le 3)$ are non-zero integers. When $\nu$ is a sufficiently 
small positive number, we write
\be{5.0} I_\nu(\bfa,\bfb)=\int_{[0,1)^3}|g(a_1\al_1)g(a_2\al_2)g(a_3\al_3)
g(b_1\al_1+b_2\al_2+b_3\al_3)|^{8-\nu} \d\bfalp .\ee
By H\"older's inequality, it follows via a change of variable and symmetry that there are non-zero integers $c_i, d_i$
$(1\le i\le 3)$ such that
$$ I_\nu(\bfa,\bfb)\ll \int_{[0,1)^3}|g(c_1\al_1)g(c_2\al_2)|^{8-2\nu} | g(c_3\al_3)g(d_1\al_1+d_2\al_2+d_3\al_3)|^{8} \d\bfalp
.$$
Since $|g(\tet)|=|g(-\tet)|$ and
$$ |g(\theta)|^8 = \sum_{|n|\le 4P^4} \psi_0(n)e(n\theta), $$
we find that 
\begin{align*}
I_\nu(\bfa,\bfb) &\ll \sum_{|n|\le 4P^4} \psi_0(n)\int_{[0,1)^3}|g(c_1\al_1)g(c_2\al_2)|^{8-2\nu}
|g(c_3\al_3)|^8 e(-n\bfd\cdot\bfalp)\d\bfalp \\
& = \sum_{|n|\le 4P^4} \psi_0(n)\psi_{2\nu}(nd_1;c_1)\psi_{2\nu}(nd_2;c_2)\psi_{0}(nd_3;c_3).
\end{align*}
Thus, on employing the inequality $|z_1z_2|\le 2(|z_1|^2+|z_2|^2)$, we obtain
\begin{align*}
I_\nu(\bfa,\bfb) &\ll \sum_{|n|\le 4P^4} \big(\psi_0(n)^2+ \psi_{0}(nd_3;c_3)^2\big) 
\big(\psi_{2\nu}(nd_1;c_1)^2+\psi_{2\nu}(nd_2;c_2)^2\big)\\
&\ll  \sum_{|n|\le 4P^4} \psi_0(k_1n;l_1)^2 \psi_{2\nu}(k_2 n;l_2)^2,
\end{align*}
for suitable non-zero integers $k_i$, $l_i$. Hence, we conclude from Lemma 
\ref{lemma4.2}  that
$$ I_\nu(\bfa,\bfb)
\ll  \sum_{|n|\le 4P^4} \psi_0(k_1n)^2 \psi_{2\nu}(k_2 n)^2 \ll P^{20-4\nu}. $$
This completes the proof of Theorem \ref{theorem1.2}

\section{Prelude to the circle method}
We assume the hypotheses of Theorem \ref{theorem1.1}, and in particular suppose that 
$s\ge 32$. With the column vectors $(a_{ij})_{1\le i\le 3} \in 
\dbZ^3\setminus  \{\mathbf 0\}$, we associate the ternary forms
$$\Lam_j=\sum_{i=1}^3 a_{ij}\alp_i\quad (1\le j\le s),$$
and the  linear forms $L_i(\bfgam)$ $(1\le i \le 3)$ defined for $\bfgam\in \dbR^s$ by
$$ L_i(\bfgam)=\sum_{j=1}^sa_{ij}\gam_j.$$

The hypotheses of Theorem \ref{theorem1.1} ensure that there is a non-singular real 
solution of the system (\ref{1.1}). By invoking homogeneity, therefore, one finds that 
there exists a real solution $\bfx=\bftet$ in $[0,1)^s$ for which the $3\times s$ matrix 
$(4a_{l,j}\theta_j^3)$ has maximal rank. Hence, there exist distinct indices $j_1$, $j_2$ 
and $j_3$ for which the $3\times 3$ matrix formed with the columns indexed by $j_1$, 
$j_2$ and $j_3$ is non-singular. The solution set of the system of equations (\ref{1.1}) 
remains unchanged if one replaces any one of its equations by the equation obtained by 
adding to it any multiple of another equation. Thus, by appropriate elementary row 
operations on the matrix of coefficients $(a_{l,j})$, there is no loss of generality in 
supposing that the system \rf{1.1} takes the form
\begin{equation}\label{6.0}
a_{l,j_l}x_{j_l}^4=-\sum_{\substack{j=1\\ j\not\in \{j_1,j_2,j_3\}}}^s a_{l,j}x_j^4\quad 
(1\le l\le 3),
\end{equation}
with $a_{l,j_l}\neq 0$ $ (1\le l\le 3)$.  
An application of the inverse function theorem consequently confirms that whenever 
$\Del>0$ is sufficiently small, the simultaneous equations
$$  a_{l,j_l} x_{j_l}^4 = - \sum_{\substack{j=1\\ j\not\in \{j_1,j_2,j_3\}}}^s a_{l,j}
(\theta_j+\Del)^4 \quad (1\le l\le 3)$$
remain soluble for $x_{l,j_l}$ with $x_{l,j_l}>0$. In this way we see that the system
 (\ref{1.1}) possesses a non-singular real solution $\bftet$ satisfying $\bftet\in (0,1)^s$. 
Now we choose a positive number $\del$ with the property that $\bftet\in (\del,1)^s$, 
and fix this value of $\del$ throughout the remaining sections of this paper. In addition, we
 fix $\eta>0$ and $\nu>0$ to be sufficiently small in the context of Theorem 
\ref{theorem1.2}.\par

Next, define 
$$G_0(\bfalp)=\prod_{j=1}^{32}g(\Lam_j)\quad \text{and}\quad  
G(\bfalp)=\prod_{j=1}^sg(\Lam_j).$$
Here and later, we write $g(\al)=g_\del(\al;P,R)$.
By orthogonality, one has
$$\calN(P)\ge \int_{[0,1)^3} G(\bfalp)\d\bfalp.$$
 The Hardy-Littlewood dissection is defined as follows.  We put $L=\log\log P$, take  
 $Q=L^{40}$, and when $b_l\in \dbZ$ $(1\le l\le 3)$ and $q\in \dbN$ we define
$$\grN(q,\bfb)=\{\bfalp\in [0,1)^3:|\alp_l-b_l /q|\le QP^{-4}\, (1\le l\le 3)\}.$$
We then take $\grN$ to be the union of the boxes $\grN(q,\bfb)$ with $0\le \bfb\le q\le Q$ 
and $(q,\bfb)=1$. Finally, we put $\grn=[0,1)^3\setminus \grN$.\par

The contribution of the major arcs $\grN$ in this dissection satisfies
\begin{equation}\label{6.4}
\int_\grN G(\bfalp)\d\bfalp \gg P^{s-12},
\end{equation}
a fact we confirm in \S8. Meanwhile, in \S7 we show that
\begin{equation}\label{6.5}
\int_\grn G(\bfalp)\d\bfalp=o(P^{s-12}).
\end{equation}
The desired conclusion $\calN(P)\gg P^{s-12}$ is immediate from (\ref{6.4}) and 
(\ref{6.5}) on noting that $[0,1)^3$ is the disjoint union of $\grN$ and $\grn$.

\section{The minor arc treatment} In this section we establish the minor arcs bound 
\rf{6.5}. We start with an inspection of the proof of \cite[Lemma 8.1]{Freeasy}. This shows that there 
exist positive numbers $B$ and $C$ with the following property. Suppose that $P$ is a 
large real number, and that $\gm$ is a real number with $P^{-B}<\gm\le 1$. Then, 
whenever $|g(\al)|\ge \gm P$, there exist integers $a$ and $q$ with
$$ (a,q)=1, \quad 1\le q\le C\gm^{-12}\quad {\text and}\quad |q\al -a|\le 
C\gm^{-12}P^{-4}. $$
Note that, whenever $|G(\bfalp)|\ge P^sL^{-1}$, then $|g(\Lam_j)|\ge PL^{-1}$ for 
$1\le j\le s$. Hence, there exist integers $c_j$ and $q_j$ with $1\le q_j \le L^{13}$, 
$(c_j,q_j)=1$ and $|q_j\Lam_j -c_j|\le L^{13}P^{-4}$ $(1\le j\le s)$. By considering the 
indices $j_1,j_2,j_3$, one finds that there exist $b_l\in\ZZ$ $(1\le l\le 3)$ and $q\in\NN$ 
with $0\le \bfb\le q\le L^{40}$, $(q,\bfb)=1$ and $|\al_l -b_l/q|\le L^{40}P^{-4}$ 
$(1\le l\le 3)$. Hence $\bfalp\in\grN$. This shows that
$$ \sup_{\bfalp\in\grn} |G(\bfalp)| \ll P^{s}L^{-1}. $$
On applying a trivial estimate for excessive factors $g(\al)$, therefore, we obtain
\begin{align*} \int_\grn G(\bfalp) \d\bfalp & \ll \Bigl( \sup_{\bfalp\in\grn} |G(\bfalp)|\Bigr)^\nu 
\int_{[0,1)^3} |G(\bfalp)|^{1-\nu}\d\bfalp \\
& \ll(P^{s-32})^{1-\nu}(P^sL^{-1})^\nu  \int_{[0,1)^3} |G_0(\bfalp)|^{1-\nu}\d\bfalp.
\end{align*}
Further, by applying H\"older's inequality, we obtain
$$ \int_{[0,1)^3} |G_0(\bfalp)|^{1-\nu}\d\bfalp
\le \prod_{l=0}^7 \Big(  \int_{[0,1)^3} |g(\Lam_{4l+1})\ldots 
g(\Lam_{4l+4})|^{8-8\nu}\d\bfalp\Big)^{1/8}. $$
On noting that $g(\al)$ has period 1, a change of variables confirms that 
for each $l$ with $0\le l\le 7$ there are non-zero integers $a_i,b_i$ $(1\le i\le 3)$ such 
that, in the notation introduced in \rf{5.0}, one has
$$   \int_{[0,1)^3} |g(\Lam_{4l+1})\ldots g(\Lam_{4l+4})|^{8-8\nu}\d\bfalp \ll 
I_{8\nu} (\bfa,\bfb). $$
Hence, by Theorem  \ref{theorem1.2}, one concludes that
$$\int_\grn G(\bfalp) \d\bfalp\ll (P^{s-32})^{1-\nu}(P^sL^{-1})^\nu (P^{20-32\nu}) \ll 
P^{s-12} L^{-\nu}.  $$
This inequality is a quantitative form of \rf{6.5}

\section{The major arcs analysis}
The analysis of the major arcs is largely standard.  Define
$$S(q,a)=\sum_{r=1}^qe(ar^4/q),\quad T(q,\bfc)=q^{-s}\prod_{j=1}^s
S(q,\Lam_j(\bfc)),$$
$$\grA(q)=\underset{(q,\bfc)=1}{\sum_{1\le \bfc\le q}}T(q,\bfc)\quad \text{and}\quad 
\grS(X)=\sum_{1\le q\le X}\grA(q).$$
Also, put
$$v(\tet)=\int_{\del P}^Pe(\tet \gam^4)\d\gam \quad \text{and}\quad 
V(\bfgam)=\prod_{j=1}^sv(\Lam_j(\bfgam)).$$
Write $\calB(X)=[-XP^{-4},XP^{-4}]^3$ and define 
$$\grJ(X)=\int_{\calB(X)}V(\bfgam)\d\bfgam.$$

Standard arguments (\cite[Lemma 5.4]{Vau1989} and \cite[Lemma 8.5]{SAEii}) show  
that there is a positive number $\rho$ having the property that whenever 
$\bfalp\in \grN(q,\bfb)\subseteq \grN$, one has
$$G(\bfalp)-\rho T(q,\bfc)V(\bfalp-\bfb/q)\ll P^s(\log P)^{-1/2}.$$
Integrating over $\grN$, we infer that
\begin{equation}\label{8.1}
\int_\grN G(\bfalp)\d\bfalp =\rho \grS(Q)\grJ(Q)+O(P^{s-12}(\log P)^{-1/4}).
\end{equation}

\begin{lemma}\label{lemma8.1} Under the hypotheses of Theorem \ref{theorem1.1}, the 
limit $\grS=\underset{X\rightarrow \infty}\lim\grS(X)$ exists, one has 
$\grS-\grS(X)\ll X^{-1/2}$, and $\grS\gg 1$.
\end{lemma}

\begin{proof} Recall that \cite[Theorem 4.2]{Vau1997} gives 
$q^{-1}S(q,a) \ll q^{-1/4}(q,a)^{1/4}$. Hence, on writing $u_j=(q,\Lam_j(\bfc))$, we 
obtain
$$T(q,\bfc)\ll q^{-8}(u_1u_2\ldots u_{32})^{1/4}.$$
By applying the elementary inequality $|z_1\ldots z_n| \le |z_1|^n + \ldots + |z_n|^n$ twice, one finds that
$$ (u_1u_2\ldots u_{32})^{1/4}
\ll  \sum_{l=0}^7 \sum_{(\gra,\grb,\grc)\in{\cal S}_l} |u_\gra u_\grb u_\grc|^{8/3}, $$
where ${\cal S}_l$ denotes the set of triples of integers $(\gra, \grb, \grc)$ with
$$4l<\gra<\grb<\grc \le 4l+4.$$
Thus
$$ \grA(q) \ll q^{-8} \max_{0\le l\le 7} \max_{(\gra,\grb,\grc)\in{\cal S}_l} 
\sum_{\substack{1\le\bfc\le q\\ (q,\bfc)=1}}|u_\gra u_\grb u_\grc|^{8/3}. $$
By symmetry, we may suppose that the maximum here occurs when $l=0$ and $(\gra,\grb,\grc)=(1,2,3)$. The argument following from equation (95)
to the end of the proof of Lemma 23 in Davenport and Lewis \cite{DL1969} then shows that
$$ \grA(q) \ll q^{-8} \underset{(u_1,u_2,u_3)\ll 1}{ \sum_{u_1\mid q} \sum_{u_2\mid q}\sum_{u_3\mid q}} (u_1u_2u_3)^{8/3} q^3/(u_1u_2u_3). $$
Since $u_1u_2u_3\ll q^2$, an elementary estimate for the divisor function yields the bound 
$\grA(q) \ll q^{\eps-5/3}$.
Hence $\underset{X\rightarrow \infty}\lim\grS(X)$ exists, and   $\grS-\grS(X)\ll X^{\eps-2/3}$.
The remaining conclusions follow as in \cite[Lemma 31]{DL1969}.
\end{proof}

\begin{lemma}\label{lemma8.2} Under the hypotheses of Theorem \ref{theorem1.1}, the 
limit $\grJ=\underset{X\rightarrow \infty }\lim \grJ(X)$ exists, one has 
$\grJ-\grJ(X)\ll P^{s-12}X^{-1}$, and $\grJ\gg P^{s-12}$.
\end{lemma}

\begin{proof} Write $\calBhat(X)$ for $\dbR^3\setminus \calB(X)$, and recall the 
prearrangement of indices implicit in (\ref{6.0}). Then a direct modification of the 
argument of \cite[Lemma 30]{DL1969}, following an analysis similar to that of Lemma 
\ref{lemma8.1}, confirms that, for a suitable positive number $\Tet=\Tet(\bfc)$, one has
$$\int_{\calBhat(X)} |v(\Lam_1(\bfgam))\ldots v(\Lam_{32}(\bfgam))|\d\bfgam \ll P^{32}\int_{\calBhat(\Tet X)} 
\prod_{i=1}^3 (1+P^4|\xi_i|)^{-8/3}\d\bfxi .$$
By applying trivial bounds for the additional factors $v(\Lam_j(\bfgam))$ for $j>32$, we therefore conclude that
$$\int_{\dbR^3\setminus \calB(X)}V(\bfxi)\d\bfxi  \ll P^{s-32}(P^{20}X^{-1})\ll P^{s-12}X^{-1}.$$
In particular, the limit $\grJ=\underset{X\rightarrow \infty }\lim \grJ(X)$ exists, and one has $\grJ-\grJ(X)\ll P^{s-12}X^{-1}$.
By the argument concluding the proof of \cite[Lemma 30]{DL1969}, one finds via Fourier's integral theorem that $\grJ\gg P^{s-12}$.
\end{proof}

Subject to the hypotheses of Theorem \ref{theorem1.1}, the conclusions of Lemmata 
\ref{lemma8.1} and \ref{lemma8.2} combine with \rf{8.1} to deliver the lower bound \rf{6.4}. 
 In view of the discussion concluding \S6,  this establishes Theorem \ref{theorem1.1}.

\section{Appendix: an adjuvant lemma} Before announcing our adjuvant pruning lemma, 
for $k\ge 4$ we define the multiplicative function $w_k(q)$ by defining, for each prime number $p$,
$$w_k(p^{uk+v})=\begin{cases} kp^{-u-1/2},&\text{when $u\ge 0$ and $v=1$,}\\
p^{-u-1},&\text{when $u\ge 0$ and $2\le v\le k$.}\end{cases}$$

\begin{lemma}\label{lemmahash1}
Suppose that $k\ge 4$. Let $\grK$ denote the union of the intervals
$$\grK(q,a)=\{\alp \in [0,1):|q\alp -a|\le P^{1-k}\},$$
with $0\le a\le q\le P$ and $(a,q)=1$. Let $\ome$ be a real number with $\ome>1$, and 
define the function $\Ups_\ome(\alp)$ for $\alp\in \grK$ by taking
$$\Ups_\ome (\alp)=w_k(q)^{2\ome}(1+P^k|\alp-a/q|)^{-\ome},$$
when $\alp \in \grK(q,a)\subseteq \grK$. Also, let $t$ be a real number with 
$t\ge \lfloor k/2\rfloor$. Then for any subset $\calA$ of $[1,P]\cap \dbZ$, one has
$$\int_\grK \Ups_\ome (\alp)\biggl| \sum_{x\in \calA}e(\alp x^k)\biggr|^{2t}\d\alp \ll 
P^{2t-k}.$$
\end{lemma}

\begin{proof} We follow the proof of \cite[Lemma 5.4]{VW2000} as far as 
\cite[equation (5.8)]{VW2000}, mutatis mutandis, reaching the estimate
\begin{equation}\label{hash1}
\int_\grK \Ups_\ome (\alp)\biggl| \sum_{x\in \calA}e(\alp x^k)\biggr|^{2t}\d\alp \ll 
P^{2t-k}\sum_{1\le q\le P}w_k(q)^{2\ome}\sig(q),
\end{equation}
where
$$\sig(q)=\sum_{r|q}rw_k(r)^{2t},$$
given in \cite[equation (5.9)]{VW2000}. Following the argument concluding the proof of 
\cite[Lemma 5.4]{VW2000}, we find that
\begin{align*}
w_k(p)^{2\ome}\sig(p)&\ll_k p^{-\ome},\\
w_k(p^{uk+1})^{2\ome}\sig(p^{uk+1})&\ll_k p^{-u-\ome+1/k}\quad (u\ge 1),\\
w_k(p^{uk+v})^{2\ome}\sig(p^{uk+v})&\ll_k p^{-u-\ome}\quad 
(\text{$u\ge 0$ and $2\le v\le k$}),
\end{align*}
whence
$$\sum_{1\le q\le P}w_k(q)^{2\ome}\sig(q)\le \prod_{p\le P}(1+Ap^{-\ome}),$$
for a suitable $A=A(k,\ome)>0$. Since $\ome >1$, the desired conclusion now follows 
from (\ref{hash1}).
\end{proof}

We apply this lemma when $k=4$ to confirm the following estimate that we announce in 
the notation of \S4. In particular, we recall the definition of the major arcs $\grM$ given 
via (\ref{Major}).

\begin{lemma}\label{lemmahash}
One has
$$\int_\grM |f(\alp)^2g(\alp)^6|d\alp \ll P^4.$$
\end{lemma}

\begin{proof} By reference to \cite[Theorem 4.1]{Vau1997} and its sequel in \cite{Vau1997}, one finds 
that when $\alp\in \grM(q,a)\subseteq \grM$, then
\begin{align*}
f(\alp)&\ll Pw_4(q)(1+P^4|\alp-a/q|)^{-1}+O(P^{1/4+\eps})\\
&\ll Pw_4(q)(1+P^4|\alp-a/q|)^{-1}.
\end{align*}
Hence, as an immediate consequence of Lemma \ref{lemmahash1}, one obtains
$$\int_\grM |f(\alp)^{8/3}g(\alp)^4|\d\alp \ll P^{8/3}\int_\grM \Ups_{4/3}(\alp)
|g(\alp)|^4\d\alp \ll P^{8/3}.$$
We recall that Lemma \ref{lemma2.1} shows $\Del_{12}=0$ to be an admissible 
exponent. Then it follows via H\"older's inequality that
\begin{align*}
\int_\grM |f(\alp)^2g(\alp)^6|\d\alp &\le \biggl( \int_\grM |f(\alp)^{8/3}g(\alp)^4|\d\alp
 \biggr)^{3/4}\biggl( \int_0^1 |g(\alp)|^{12}\d\alp \biggr)^{1/4}\\
 &\ll \left( P^{8/3}\right)^{3/4}\left( P^{8}\right)^{1/4}\ll P^4.
\end{align*}
This completes the proof of the lemma.
\end{proof}

\frenchspacing
\bibliographystyle{amsbracket}
\providecommand{\bysame}{\leavevmode\hbox to3em{\hrulefill}\thinspace}

\end{document}